\documentclass[10pt,reqno]{amsart}
\usepackage{eurosym}
\usepackage{amsmath, amsthm, amscd, amsfonts, amssymb, graphicx, color}
\usepackage[bookmarksnumbered, colorlinks, plainpages]{hyperref}

\newtheorem{theorem}{Theorem}[section]
\newtheorem{lemma}[theorem]{Lemma}

\newtheorem{corollary}[theorem]{Corollary}
\theoremstyle{definition}
\newtheorem{definition}[theorem]{Definition}

\theoremstyle{remark}

\numberwithin{equation}{section}
\input{tcilatex}

\begin{document}
\title[A study of multivalent $q$-starlike functions]{A study of multivalent 
$q$-starlike functions connected with circular domain}
\author[L. Shi, Q. Khan, G. Srivastava, J.-L. Liu, M. Arif]{Lei Shi$^{1}$,
Qaiser Khan$^{2}$, Gautam Srivastava$^{3,4}$, Jin-Lin Liu$^{5}$, Muhammad
Arif$^{2,\ast}$}
\address{$^{1}$School of Mathematics and Statistics, Anyang Normal
University, Anyan 455002, Henan, People's Republic of China.}
\email{shimath@163.com}
\address{$^{2}$Department of Mathematics, Abdul Wali Khan University Mardan,
Mardan 23200, KP, Pakistan}
\email{qaisermath84@gmail.com, marifmaths@awkum.edu.pk}
\address{$^{3}$Department of Mathematics and Computer Science, Brandon
University, Brandon, MB R7A 6A9, Canada\\
$^{4}$Research Center for Interneural Computing, China Medical University,
Taichung 40402, Taiwan}
\email{srivastavag@brandonu.ca}
\address{$^{5}$Department of Mathematics, Yangzhou University, Yangzhou
225002, People's Republic of China}
\email{jlliu@yzu.edu.cn}
\date{16 May 2019 \\
\indent$^{*}$Corresponding author\\
2010\textit{\ Mathematics Subject Classification.} 30C45, 30C50.}

\begin{abstract}
In the present article, our aim is to examine some useful problems including
the convolution problem, sufficiency criteria, coefficient estimates and
Fekete-Szeg\"{o} type inequalities for a new subfamily of analytic and
multivalent functions associated with circular domain. In addition, we also
define and study a Bernardi integral operator in its $q$-extension for
multivalent functions.
\end{abstract}

\maketitle

\section{Introduction}

\noindent The study of the $q$-extension of calculus or the $q$-analysis
attracted and motivated many researchers becauuse of its applications in
different parts of mathematical sciences. Jackson \cite{Jackson2,Jackson1})
was one of the main contributer among all the mathematicians who initiated
and established the theory of $q$-calculus. As an interesting sequel to \cite%
{Ismail}, in which use was made of the $q$-derivative operator for the first
time for studying the geometry of $q$-starlike functions, a firm footing of
the usage of the $q$-calculus in the context of Geometric Function Theory
was actually provided and the basic (or $q$-) hypergeometric functions were
first used in Geometric Function Theory in a book chapter by Srivastava
(see, for details, \cite[pp. 347 \textit{et seq.}]{b}). The theory of $q$%
-starlike functions was later extended to various families of $q$-starlike
functions by (for example) Agrawal and Sahoo \cite{Sahoo} (see also the
recent investigations on this subject by Srivastava et al. \cite%
{HMS1,HMS3,HMS4,HMS5,HMS6,HMS7}). Motivated by these $q$-developments in
Geometric Function Theory, many authors such as like Srivastava and Bansal 
\cite{b} were added their contributions in this direction which has made
this research area much more attractive.

In 2014, Kanas and R\u{a}ducanu \cite{Kanasq} used the familiar Hadamad
product to define a $q$-extension of the Ruscheweyh operator and discussed
important applications of this operator. Moreover, the extensive study of
this $q$-Ruscheweyh operator was further made by Mohammad and Darus \cite%
{Maslina2} and Mahmood and Sok\'{o}{\l } \cite{Shahid}. Recently, a new idea
was presented by Darus \cite{Maslina1} and introduced a new differential
operator called generalized $q$-differential operator with the help of $q$%
-hypergeometric functions where they studied some useful applications of
this operator. For the recent extension of different operators in $q$%
-analogue, see the references \cite{bak1,arif1,arif2}. The operator defined
in \cite{Kanasq} was extended further for multivalent functions by Arif et
al. \cite{arif} in which they investigated its important applications. The
aim of this paper is to define a family of multivalent $q$-starlike
functions associated with circular domain and to study some of its useful
properties. \medskip 

\noindent Let $\mathfrak{A}_{p}$ $\left( p\in 
\mathbb{N}
=\left \{ 0,1,2,\ldots \right \} \right) $ contains all multivalent
functions say $f$ that are holomorphic or analytic in a subset $\mathbb{D}%
=\left \{ z:\left \vert z\right \vert <1\right \} $ of a complex plane $%
\mathbb{C}
$ and having the series form:%
\begin{equation}
f(z)=z^{p}+\sum \limits_{l=1}^{\infty }a_{l+p}z^{l+p},\text{ }\left( z\in 
\mathbb{D}\right) .  \label{eq1}
\end{equation}%
For two analytic functions $f$ and $g$ in $\mathbb{D},$ then $f$ is
subordinate to $g,$ symbolically presented as $f\prec g$ or $f\left(
z\right) \prec g\left( z\right) ,$ if we can find an analytic function $w$
with the properties $w\left( 0\right) =0$ \& $\left \vert w\left( z\right)
\right \vert <1$ such that $f(z)=g(w(z))$ $\  \left( z\in \mathbb{D}\right) .$
Also, if $g$ is univalent in $\mathbb{D},$ then we have:%
\begin{equation*}
\begin{tabular}{llllll}
$f(z)\prec g(z)$ & $(z\in \mathbb{D})$ & $\Longleftrightarrow $ & $f(0)=g(0)$
& and & $f(\mathbb{D})\subset g(\mathbb{D}).$%
\end{tabular}%
\end{equation*}%
For given $q\in \left( 0,1\right) $, the derivative in $q$-analogue of $f$
is given by 
\begin{equation}
\mathcal{\partial }_{q}f(z)=\frac{f(z)-f\left( qz\right) }{z\left(
1-q\right) },\text{ }\left( z\neq 0,\text{ }q\neq 1\right) .  \label{aa2}
\end{equation}%
Making $\left( \ref{eq1}\right) $ and $\left( \ref{aa2}\right) ,$ we easily
get that for $n\in \mathbb{N}$ and $z\in \mathbb{D}$ 
\begin{equation}
\mathcal{\partial }_{q}\left \{ \sum_{n=1}^{\infty }a_{n+p}z^{n+p}\right \}
=\sum_{n=1}^{\infty }\left[ n+p,q\right] a_{n+p}z^{n+p-1},  \label{l1}
\end{equation}%
where 
\begin{equation*}
\left[ n,q\right] =\frac{1-q^{n}}{1-q}=1+\sum \limits_{l=1}^{n-1}q^{l},\  \ %
\left[ 0,q\right] =0.
\end{equation*}%
For $n\in 
\mathbb{Z}
^{\ast }:=%
\mathbb{Z}
\backslash \left \{ -1,-2,\ldots \right \} ,$ the $q$-number shift factorial
is given as: 
\begin{equation*}
\left[ n,q\right] !=\left \{ 
\begin{array}{l}
1,\text{ }n=0, \\ 
\left[ 1,q\right] \left[ 2,q\right] \ldots \left[ n,q\right] ,\text{ }n\in 
\mathbb{N}
.%
\end{array}%
\right.
\end{equation*}%
Also, with $x>0$, the $q$-analogue of the Pochhammer symbol has the form: 
\begin{equation*}
\lbrack x,q]_{n}=\left \{ 
\begin{array}{l}
1,\ n=0, \\ 
\lbrack x,q][x+1,q]\cdots \lbrack x+n-1,q],\text{ }n\in 
\mathbb{N}
,%
\end{array}%
\right.
\end{equation*}%
and, for $x>0$, the Gamma function in $q$-analogue is presented as 
\begin{equation*}
\Gamma _{q}\left( x+1\right) =\left[ x,q\right] \Gamma _{q}\left( t\right) 
\text{ and }\Gamma _{q}\left( 1\right) =1.
\end{equation*}%
We now consider a function 
\begin{equation}
\Phi _{p}\left( q,\mu +1;z\right) =z^{p}+\sum_{n=2}^{\infty }\Lambda _{n+p}%
\text{ }z^{n+p},\text{ }(\mu >-1,\text{ }z\in \mathbb{D}),  \label{pf}
\end{equation}%
with 
\begin{equation}
\Lambda _{n+p}=\frac{[\mu +1,q]_{n+p}}{[n+p,q]!}.  \label{ita}
\end{equation}

\noindent The series defined in $\left( \ref{pf}\right) $ is converges
absolutely in $\mathbb{D}$. Using $\Phi _{p}\left( q,\mu ;z\right) $ with $%
\mu >-1$ and idea of convolution, Arif et.al \cite{arif} established a
differential operator $\mathcal{L}_{q}^{\mu +p-1}:\mathfrak{A}%
_{p}\rightarrow \mathfrak{A}_{p}$ by%
\begin{equation}
\mathcal{L}_{q}^{\mu +p-1}f\left( z\right) =\Phi _{p}\left( q,\mu ;z\right)
\ast f(z)=z^{p}+\sum_{n=2}^{\infty }\Lambda _{n+p}\text{ }a_{n+p}z^{n+p},%
\text{ }\left( z\in \mathbb{D}\right) ,  \label{oper2}
\end{equation}%
Also we note that 
\begin{equation*}
\lim \limits_{q\rightarrow 1^{-}}\Phi _{p}\left( q,\mu ;z\right) =\frac{z^{p}%
}{\left( 1-z\right) ^{\mu +1}}\text{ \ and \ }\lim \limits_{q\rightarrow
1^{-}}\mathcal{L}_{q}^{\mu +p-1}f(z)=f(z)\ast \frac{z^{p}}{\left( 1-z\right)
^{\mu +1}}.
\end{equation*}

\noindent Now when $q\rightarrow 1^{-},$ the operator defined in $\left( \ref%
{oper2}\right) $ becomes to the familiar differential operator investigated
in \cite{Goel} and further, setting $p=1,$ we get the most familiar operator
known as Ruscheweyh operator \cite{R} (see also \cite{Ahuja,noor}). Also,
for different types of operators in $q$-analogue, see the works \cite%
{bak1,Maslina4,Maslina3,arif3,arif1,Maslina1}. \bigskip

\noindent Motivated from the work studied in \cite%
{arif2,Ismail,Sh1,Aouf1,Wang}, we establish a family $\mathcal{S}_{p}^{\ast
}\left( q,\mu ,A,B\right) $ using the operator $\mathcal{L}_{q}^{\mu +p-1}$
as:

\begin{definition}
Suppose that $q\in \left( 0,1\right) $ and $-1\leqq B<A\leqq 1.$ Then $f\in 
\mathfrak{A}_{p}$ belongs to the set $\mathcal{S}_{p}^{\ast }\left( q,\mu
,A,B\right) ,$ if it satisfies 
\begin{equation}
\frac{z\partial _{q}\mathcal{L}_{q}^{\mu +p-1}f\left( z\right) }{\left[ p,q%
\right] \mathcal{L}_{q}^{\mu +p-1}f\left( z\right) }\prec \frac{1+Az}{1+Bz},
\label{eq6}
\end{equation}%
where the function $\frac{1+Az}{1+Bz}$ is known as Janowski function studied
in \cite{Janwoski}.
\end{definition}

\noindent Alternatively,%
\begin{equation}
f\in \mathcal{S}_{p}^{\ast }\left( q,\mu ,A,B\right) \Leftrightarrow \left
\vert \frac{\frac{z\partial _{q}\mathcal{L}_{q}^{\mu +p-1}f\left( z\right) }{%
\left[ p,q\right] \mathcal{L}_{q}^{\mu +p-1}f\left( z\right) }-1}{A-B\frac{%
z\partial _{q}\mathcal{L}_{q}^{\mu +p-1}f\left( z\right) }{\left[ p,q\right] 
\mathcal{L}_{q}^{\mu +p-1}f\left( z\right) }}\right \vert <1.  \label{eq7}
\end{equation}

\section{\textsc{A Set of Lemmas}}

\begin{lemma}
\cite{Rogosinski} \label{R}Let $h\left( z\right) =1+\sum
\limits_{n=1}^{\infty }d_{n}z^{n}\prec K\left( z\right) =1+\sum
\limits_{n=1}^{\infty }k_{n}z^{n}$ in $\mathbb{D}$. If $K\left( z\right) $
is convex univalent in $\mathbb{D},$ then%
\begin{equation*}
\left \vert d_{n}\right \vert \leqq \left \vert k_{1}\right \vert ,\text{ \
for }n\geqq 1.
\end{equation*}
\end{lemma}

\begin{lemma}
\label{lemm2} Let $\mathcal{W}$ contain all functions $w$ that are analytic
in $\mathbb{D}$ which satisfying $w\left( 0\right) =0$ \& $\left \vert
w(z)\right \vert <1.$ If the function $w\in \mathcal{W}$ given by 
\begin{equation*}
w(z)=\sum \limits_{k=1}^{\infty }w_{k}z^{k}\text{ }\left( z\in \mathbb{D}%
\right) .
\end{equation*}%
then for $\lambda \in 
\mathbb{C}
,$ we have 
\begin{equation}
\left \vert w_{2}-\lambda w_{1}^{2}\right \vert \leqq \max \left \{ 1;\text{ 
}|\lambda |\right \} ,  \label{L21}
\end{equation}%
and 
\begin{equation}
\left \vert w_{3}+\frac{1}{4}w_{1}w_{2}+\frac{1}{16}w_{1}^{3}\right \vert
\leqq 1.  \label{L22}
\end{equation}%
These results are best possible.
\end{lemma}

\noindent For the first and second part see the reference \cite{keough} and 
\cite{sokol} respectively.

\section{\textsc{The Main Results and Their Consequences}}

\begin{theorem}
Let $f\in \mathfrak{A}_{p}$ has the series form $\left( \ref{eq1}\right) $
and satisfing the inequality given by 
\begin{equation}
\sum_{n=1}^{\infty }\wedge _{n+p}\left( \left[ n+p,q\right] \left(
1-B\right) -\left[ p,q\right] \left( 1-A\right) \right) \left \vert
a_{n+p}\right \vert \leqq \left[ p,q\right] \left( A-B\right) .  \label{eq11}
\end{equation}
Then $f\in \mathcal{S}_{p}^{\ast }\left( q,\mu ,A,B\right) .$
\end{theorem}

\begin{proof}
To show $f\in \mathcal{S}_{p}^{\ast }\left( q,\mu ,A,B\right) ,$ we just
need to show the relation $\left( \ref{eq7}\right) $. For this we consider 
\begin{equation*}
\left \vert \frac{\frac{z\partial _{q}\mathcal{L}_{q}^{\mu +p-1}f\left(
z\right) }{\left[ p,q\right] \mathcal{L}_{q}^{\mu +p-1}f\left( z\right) }-1}{%
A-B\frac{z\partial _{q}\mathcal{L}_{q}^{\mu +p-1}f\left( z\right) }{\left[
p,q\right] \mathcal{L}_{q}^{\mu +p-1}f\left( z\right) }}\right \vert =.\left
\vert \frac{z\partial _{q}\mathcal{L}_{q}^{\mu +p-1}f\left( z\right) -\left[
p,q\right] \mathcal{L}_{q}^{\mu +p-1}f\left( z\right) }{A\left[ p,q\right] 
\mathcal{L}_{q}^{\mu +p-1}f\left( z\right) -Bz\partial _{q}\mathcal{L}%
_{q}^{\mu +p-1}f\left( z\right) }\right \vert
\end{equation*}%
Using $\left(\ref{oper2}\right)$, and then with the help of $\left( \ref%
{eq11}\right)$ and $\left(\ref{l1}\right) $, we have 
\begin{eqnarray*}
&=&%
\begin{array}{l}
\left \vert \frac{\left[ p,q\right] z^{p}+\sum_{n=1}^{\infty }\wedge
_{n+p}a_{n+p}\left[ n+p,q\right] z^{n+p}-\left[ p,q\right] \left(
z^{p}+\sum_{n=1}^{\infty }\wedge _{n+p}a_{n+p}z^{n+p}\right) }{A\left[ p,q%
\right] \left( z^{p}+\sum_{n=1}^{\infty }\wedge _{n+p}a_{n+p}z^{n+p}\right)
-B\left( \left[ p,q\right] z^{p}+\sum_{n=1}^{\infty }\wedge _{n+p}a_{n+p}%
\left[ n+p,q\right] z^{n+p}\right) }\right \vert%
\end{array}
\\
&=&%
\begin{array}{l}
\left \vert \frac{\sum_{n=1}^{\infty }\wedge _{n+p}a_{n+P}\left( \left[ n+p,q%
\right] -\left[ p,q\right] \right) z^{n+p}}{\left( A-B\right) \left[ p,q%
\right] z^{p}+\sum_{n=1}^{\infty }\wedge _{n+p}a_{n+p}\left( A\left[ p,q%
\right] -B\left[ n+p,q\right] \right) z^{n+p}}\right \vert%
\end{array}
\\
&\leqq &%
\begin{array}{l}
\frac{\sum_{n=1}^{\infty }\wedge _{n+p}\left \vert a_{n+P}\right \vert
\left( \left[ n+p,q\right] -\left[ p,q\right] \right) \left \vert z\right
\vert ^{n+p}}{\left( A-B\right) \left[ p,q\right] \left \vert z\right \vert
^{p}-\sum_{n=1}^{\infty }\wedge _{n+p}\left \vert a_{n+p}\right \vert \left(
A\left[ p,q\right] -B\left[ n+p,q\right] \right) \left \vert z\right \vert
^{n+p}}%
\end{array}
\\
&\leqq &%
\begin{array}{l}
\frac{\sum_{n=1}^{\infty }\wedge _{n+p}\left \vert a_{n+P}\right \vert
\left( \left[ n+p,q\right] -\left[ p,q\right] \right) }{\left( A-B\right) %
\left[ p,q\right] -\sum_{n=1}^{\infty }\wedge _{n+p}\left \vert
a_{n+p}\right \vert \left( A\left[ p,q\right] -B\left[ n+p,q\right] \right) }%
<1,%
\end{array}%
\end{eqnarray*}%
where we have used the inequality $\left( \ref{eq11}\right) $ and this
completes the proof.
\end{proof}

\noindent Varying the parameters $\mu ,$ $b,$ $A$ and $B$ in the last
Theorem, we get the following known results discussed earlier in \cite{Seoc}.

\begin{corollary}
\label{cor31} Let $f\in \mathfrak{A}$ be given by $\left( \ref{eq1}\right) $
and satisfy the inequality%
\begin{equation*}
\sum_{n=2}^{\infty }\left( \left[ n,q\right] (1-B)-1+A\right) \left \vert
a_{n}\right \vert \leqq A-B.
\end{equation*}%
Then the function $f\in \mathcal{S}_{q}^{\ast }[A,B].$
\end{corollary}

\noindent By choosing $q\rightarrow 1^{-}$ in the last corollary, we get the
known result proved by Ahuja \cite{Ahuja} and furthermore for $A=1-\alpha $
and $B=-1,$ we obtain the result for the family $\mathcal{S}^{\ast }\left(
\xi \right) $ which was proved by Silverman \cite{silv2}.

\begin{theorem}
\label{Th1} Let $f\in \mathcal{S}_{p}^{\ast }\left( q,\mu ,A,B\right) $ be
of the form $\left( \ref{eq1}\right) .$ Then 
\begin{equation}
\left \vert a_{p+1}\right \vert \leqq \frac{\psi _{1}\left( A-B\right) }{%
\wedge _{1+p}},  \label{eq31}
\end{equation}%
and for $n\geqq 2,$%
\begin{equation}
\left \vert a_{n+p}\right \vert \leqq \frac{\left( A-B\right) \psi _{n}}{%
\wedge _{n+p}}\prod \limits_{t=1}^{n-1}\left( 1+\frac{\left[ p,q\right]
\left( A-B\right) }{\left( \left[ p+t,q\right] -\left[ p,q\right] \right) }%
\right) ,  \label{eq32}
\end{equation}%
where%
\begin{equation}
\psi _{n}:=\psi _{n}(p,q)=\frac{\left[ p,q\right] }{\left( \left[ n+p,q%
\right] -\left[ p,q\right] \right) }.  \label{eq33}
\end{equation}
\end{theorem}

\begin{proof}
If $f\in \mathcal{S}_{p}^{\ast }\left( q,\mu ,A,B\right) ,$ then by
definition we have 
\begin{equation}
\frac{z\partial _{q}\mathcal{L}_{q}^{\mu +p-1}f\left( z\right) }{\left[ p,q%
\right] \mathcal{L}_{q}^{\mu +p-1}f\left( z\right) }=\frac{1+Aw(z)}{1+Bw(z)}.
\label{eq22}
\end{equation}%
Let us put%
\begin{equation*}
p(z)=1+\sum \limits_{n=1}^{\infty }d_{n}z^{n}=\frac{1+Aw(z)}{1+Bw(z)}.
\end{equation*}%
Then by Lemma \ref{R}, we get%
\begin{equation}
\left \vert d_{n}\right \vert \leqq A-B.  \label{l3}
\end{equation}%
Now, from $\left( \ref{eq22}\right) $ and $\left( \ref{oper2}\right) $, we
can write%
\begin{equation}
\begin{array}{lll}
z^{p}+\sum \limits_{n=1}^{\infty }\frac{\left[ n+p,q\right] }{\left[ p,q%
\right] }\Lambda _{n+p}\text{ }a_{n+p}z^{n+p} & = & \left( 1+\sum
\limits_{n=1}^{\infty }d_{n}z^{n}\right) \left( z^{p}+\sum
\limits_{n=1}^{\infty }\Lambda _{n+p}\text{ }a_{n+p}z^{n+p}\right) .%
\end{array}
\label{l2}
\end{equation}%
Equating coefficients of $z^{n+p}$ on both sides%
\begin{equation*}
\begin{array}{lll}
\wedge _{n+p}\left( \left[ n+p,q\right] -\left[ p,q\right] \right) a_{n+p} & 
= & \left[ p,q\right] \wedge _{n+p-1}a_{n+p-1}d_{1}+\cdots +\left[ p,q\right]
\wedge _{1+p}a_{1+p}d_{n-1}.%
\end{array}%
\end{equation*}%
Taking absolute on both sides and then using $\left( \ref{l3}\right) ,$ we
have 
\begin{equation*}
\wedge _{n+p}\left( \left[ n+p,q\right] -\left[ p,q\right] \right) \left
\vert a_{n+p}\right \vert \leqq \left[ p,q\right] \left( A-B\right) \left(
1+\sum \limits_{k=1}^{n-1}\wedge _{k+p}\left \vert a_{k+p}\right \vert
\right) ,
\end{equation*}%
and this further implies%
\begin{equation}
\left \vert a_{n+p}\right \vert \leqq \frac{\left( A-B\right) \psi _{n}}{%
\wedge _{n+p}}\left( 1+\sum \limits_{k=1}^{n-1}\wedge _{k+p}\left \vert
a_{k+p}\right \vert \right) ,  \label{eq13}
\end{equation}%
where $\psi _{n}$ is given by $\left( \ref{eq33}\right) .$ So for $n=1$, we
have from $\left( \ref{eq13}\right) $%
\begin{equation*}
\left \vert a_{p+1}\right \vert \leqq \frac{\left( A-B\right) \psi _{1}}{%
\wedge _{1+p}},
\end{equation*}%
and this shows that $\left( \ref{eq31}\right) $ holds for $n=1.$ To prove $%
\left( \ref{eq32}\right) $ we apply mathematical induction. Therefore for $%
n=2$, we have from $\left( \ref{eq31}\right) $

\begin{equation*}
\left \vert a_{p+2}\right \vert \leqq \frac{\left( A-B\right) \psi _{2}}{%
\wedge _{2+p}}\left( 1+\wedge _{1+p}\left \vert a_{1+p}\right \vert \right) ,
\end{equation*}%
using $\left( \ref{eq31}\right) $, we have 
\begin{equation*}
\left \vert a_{p+2}\right \vert \leqq \frac{\left( A-B\right) \psi _{2}}{%
\wedge _{2+p}}\left( 1+\left( A-B\right) \psi _{1}\right) ,
\end{equation*}%
which clearly shows that $\left( \ref{eq32}\right) $ holds for $n=2$. Let us
assume that $\left( \ref{eq32}\right) $ is true for $n\leqq m-1,$ that is,%
\begin{equation*}
\left \vert a_{m-1+p}\right \vert \leqq \frac{\left( A-B\right) \psi _{m-1}}{%
\wedge _{m+p-1}}\prod \limits_{t=1}^{m-2}\left( 1+\left( A-B\right) \psi
_{t}\right) .
\end{equation*}%
Consider%
\begin{eqnarray*}
\left \vert a_{m+p}\right \vert &\leqq &\frac{\left( A-B\right) \psi _{m}}{%
\wedge _{m+p}}\left( 1+\sum \limits_{k=1}^{m-1}\wedge _{k+p}\left \vert
a_{k+p}\right \vert \right) \\
&=&\frac{\left( A-B\right) \psi _{m}}{\wedge _{m+p}}\left \{ 1+\left(
A-B\right) \psi _{1}+\ldots +\left( A-B\right) \psi _{m-1}\prod
\limits_{t=1}^{m-2}\left( 1+\left( A-B\right) \psi _{t}\right) \right \} \\
&=&\frac{\left( A-B\right) \psi _{m}}{\wedge _{m+p}}\prod
\limits_{t=1}^{m-1}\left( 1+\frac{\left[ p,q\right] \left( A-B\right) }{%
\left( \left[ p+t,q\right] -\left[ p,q\right] \right) }\right) ,
\end{eqnarray*}
and this implies that the given result is true for $n=m.$ Hence, using
mathematical induction, we achived the inequality $\left( \ref{eq32}\right)
. $
\end{proof}

\begin{theorem}
\label{Th2}Let $f\in \mathcal{S}_{p}^{\ast }\left( q,\mu ,A,B\right) $ and
be given by $\left( \ref{eq1}\right) .$ Then for $\lambda \in \mathbb{C}$ 
\begin{equation*}
\left \vert a_{p+2}-\lambda a_{p+1}^{2}\right \vert \leqq \frac{(A-B)\psi
_{2}}{\Lambda _{p+2}}\left \{ 1;\text{ }\left \vert \upsilon \right \vert
\right \} ,
\end{equation*}%
where $\upsilon $ is given by%
\begin{equation}
\upsilon =\left( B-(A-B)\psi _{1}\right) +\frac{\Lambda _{p+2}\psi _{1}^{2}}{%
\Lambda _{p+1}^{2}\psi _{2}}(A-B)\lambda .  \label{eq34}
\end{equation}
\end{theorem}

\begin{proof}
Let $f\in \mathcal{S}_{p}^{\ast }\left( q,\mu ,A,B\right) $ and consider the
right hand side of $\left( \ref{eq22}\right) $ we have%
\begin{equation*}
\frac{1+Aw(z)}{1+Bw(z)}=\left( 1+A\sum \limits_{k=1}^{\infty
}w_{k}z^{k}\right) \left( 1+B\sum \limits_{k=1}^{\infty }w_{k}z^{k}\right)
^{-1},
\end{equation*}%
where 
\begin{equation*}
w(z)=\sum \limits_{k=1}^{\infty }w_{k}z^{k},
\end{equation*}%
and after simple computations, we can rewrite 
\begin{equation}
\frac{1+Aw(z)}{1+Bw(z)}=1+(A-B)w_{1}z+(A-B)\left \{ w_{2}-Bw_{1}^{2}\right
\} z^{2}+\ldots .  \label{Fekete1}
\end{equation}%
Now, left hand side of $\left( \ref{eq22}\right) $, we have 
\begin{eqnarray}
\frac{z\partial _{q}\mathcal{L}_{q}^{\mu +p-1}f\left( z\right) }{\left[ p,q%
\right] \mathcal{L}_{q}^{\mu +p-1}f\left( z\right) } &=&\left( 1+\sum
\limits_{n=1}^{\infty }\frac{\left[ n+p,q\right] }{\left[ p,q\right] }%
\Lambda _{n+p}\text{ }a_{n+p}z^{n}\right) \left( 1+\sum
\limits_{n=1}^{\infty }\Lambda _{n+p}\text{ }a_{n+p}z^{n}\right) ^{-1} 
\notag \\
&=&1+\frac{\Lambda _{1+p}}{\psi _{1}}a_{1+p}z+\left( \frac{\Lambda
_{2+p}a_{2+p}}{\psi _{2}}-\frac{\Lambda _{1+p}^{2}\text{ }a_{1+p}^{2}}{\psi
_{1}}\right) z^{2}+\ldots .  \label{Fekete2}
\end{eqnarray}%
From $\left( \ref{Fekete1}\right) $ and $\left( \ref{Fekete2}\right) ,$ we
have 
\begin{eqnarray*}
a_{p+1} &=&\frac{\psi _{1}}{\Lambda _{p+1}}(A-B)w_{1} \\
a_{p+2} &=&\frac{(A-B)\psi _{2}}{\Lambda _{p+2}}\left \{ w_{2}+\left(
(A-B)\psi _{1}-B\right) w_{1}^{2}\right \} .
\end{eqnarray*}%
Now consider%
\begin{eqnarray*}
\left \vert a_{p+2}-\lambda a_{p+1}^{2}\right \vert &=&\left \vert \frac{%
(A-B)\psi _{2}}{\Lambda _{p+2}}\left \{ w_{2}+\left( (A-B)\psi _{1}-B\right)
w_{1}^{2}\right \} -\lambda \frac{\psi _{1}^{2}}{\Lambda _{p+1}^{2}}%
(A-B)^{2}w_{1}^{2}\right \vert \\
&=&\frac{(A-B)\psi _{2}}{\Lambda _{p+2}}\left \vert w_{2}-\left \{ \left(
B-(A-B)\psi _{1}\right) +\frac{\Lambda _{p+2}\psi _{1}^{2}}{\Lambda
_{p+1}^{2}\psi _{2}}(A-B)\lambda \right \} w_{1}^{2}\right \vert ,
\end{eqnarray*}%
using Lemma \ref{lemm2}, we have%
\begin{equation*}
\left \vert a_{p+2}-\lambda a_{p+1}^{2}\right \vert \leqq \frac{(A-B)\psi
_{2}}{\Lambda _{p+2}}\left \{ 1;\text{ }\left \vert \upsilon \right \vert
\right \} ,
\end{equation*}%
where $\upsilon $ is given by%
\begin{equation*}
\upsilon =\left( B-(A-B)\psi _{1}\right) +\frac{\Lambda _{p+2}\psi _{1}^{2}}{%
\Lambda _{p+1}^{2}\psi _{2}}(A-B)\lambda .
\end{equation*}%
This completes the proof.
\end{proof}

\begin{theorem}
Let $f\in \mathcal{S}_{p}^{\ast }\left( q,\mu ,A,B\right) $ and be given by $%
\left( \ref{eq1}\right) .$ Then 
\begin{equation*}
\left \vert a_{p+3}-\frac{q+2}{q^{2}+q+1}\frac{\Lambda _{1+p}\Lambda _{2+p}}{%
\Lambda _{3+p}}a_{p+2}a_{p+1}+\frac{1}{[3,q]}\frac{\Lambda _{1+p}^{3}}{%
\Lambda _{3+p}}a_{p+1}^{3}\right \vert \leqq (A-B)\left \{ \frac{4\left(
2B-1\right) ^{2}+1}{8\Lambda _{3+p}}\right \} \psi _{3},
\end{equation*}%
where $\psi _{n}$ and $\wedge _{n+p}$ are defined by $\left( \ref{eq33}%
\right) $ and $\left( \ref{ita}\right) $ respectively.
\end{theorem}

\begin{proof}
From the relations $\left( \ref{Fekete1}\right) $ and $\left( \ref{Fekete2}%
\right) ,$ we have%
\begin{equation*}
\left( a_{p+3}-\frac{q+2}{q^{2}+q+1}\frac{\Lambda _{1+p}\Lambda _{2+p}}{%
\Lambda _{3+p}}a_{p+2}a_{p+1}+\frac{1}{[3,q]}\frac{\Lambda _{1+p}^{3}}{%
\Lambda _{3+p}}a_{p+1}^{3}\right) =\frac{(A-B)\psi _{3}}{\Lambda _{3+p}}%
\left \{ w_{3}-2Bw_{1}w_{2}+B^{2}w_{1}^{3}\right \} ,
\end{equation*}%
equivalently, we have%
\begin{equation*}
\left \vert \left( a_{p+3}-\frac{q+2}{q^{2}+q+1}\frac{\Lambda _{1+p}\Lambda
_{2+p}}{\Lambda _{3+p}}a_{p+2}a_{p+1}+\frac{1}{[3,q]}\frac{\Lambda _{1+p}^{3}%
}{\Lambda _{3+p}}a_{p+1}^{3}\right) \right \vert \text{ \  \  \  \  \  \  \  \  \  \
\  \  \  \  \  \  \  \  \  \  \  \  \  \  \  \  \  \  \  \  \  \  \  \  \  \  \  \  \  \  \  \  \  \  \  \  \  \
\  \  \  \  \  \  \  \  \  \  \ }
\end{equation*}%
\begin{eqnarray*}
&=&\frac{(A-B)\psi _{3}}{\Lambda _{3+p}}\left \vert \left( w_{3}+\frac{1}{4}%
w_{1}w_{2}+\frac{1}{16}w_{1}^{3}\right) -\frac{16B^{2}-1}{16}\left(
w_{2}-w_{1}^{2}\right) +\frac{16B^{2}-32B-5}{16}w_{2}\right \vert \\
&\leqq &\frac{(A-B)\psi _{3}}{\Lambda _{3+p}}\left \{ 1+\frac{16B^{2}-1}{16}+%
\frac{16B^{2}-32B-5}{16}\right \} \\
&\leqq &\frac{(A-B)\psi _{3}}{\Lambda _{3+p}}\left \{ \frac{16B^{2}-16B+5}{8}%
\right \} ,
\end{eqnarray*}%
where we have used $\left( \ref{L21}\right) $ and $\left( \ref{L22}\right) .$%
This completes the proof.
\end{proof}

\begin{theorem}
\label{Thm2} Let $f\in \mathfrak{A}_{p}$ be given by $\left( \ref{eq1}%
\right) . $ Then the function $f$ is in the class $\mathcal{S}_{p}^{\ast
}\left( q,\mu ,A,B\right) ,$ if and only if 
\begin{equation}
\frac{e^{i\theta }\left( B-\left[ p,q\right] A\right) }{z}\left[ \mathcal{L}%
_{q}^{\mu +p-1}f\left( z\right) \ast \left( \frac{\left( N+1\right)
z^{p}-qLz^{p+1}}{\left( 1-z\right) \left( 1-qz\right) }\right) \right] \neq
0,  \label{eq11a}
\end{equation}%
for all 
\begin{eqnarray}
N &=&N_{\theta }=\frac{\left( \left[ p,q\right] -1\right) e^{-i\theta }}{%
\left( \left[ p,q\right] A-B\right) },  \notag \\
&&  \label{LM1} \\
L &=&L_{\theta }=\frac{\left( e^{-i\theta }+\left[ p,q\right] A\right) }{%
\left( \left[ p,q\right] A-B\right) },  \notag
\end{eqnarray}%
and also for $N=0,$ $L=1.$
\end{theorem}

\begin{proof}
Since the function $f\in \mathcal{S}_{p}^{\ast }\left( q,\mu ,A,B\right) $
is analytic in $\mathbb{D}$, it implies that $\mathcal{L}_{q}^{\mu
+p-1}f\left( z\right) \neq 0$ for all $z\in \mathbb{D}^{\ast }=\mathbb{D}%
\backslash \{0\}$; that is 
\begin{equation*}
\frac{e^{i\theta }\left( B-\left[ p,q\right] A\right) }{z}\mathcal{L}%
_{q}^{\mu +p-1}f\left( z\right) \neq 0\text{ \ }\left( z\in \mathbb{D}%
\right) ,
\end{equation*}%
and this is equivalent to $\left( \ref{eq11a}\right) $ for $N=0$ and $L=1.$
From $\left( \ref{eq6}\right) ,$ according to the definition of the
subordination, there exists an analytic function $w$ with the property that $%
w\left( 0\right) =0$ and $\left \vert w(z)\right \vert <1$ such that%
\begin{equation*}
\frac{z\partial _{q}\mathcal{L}_{q}^{\mu +p-1}f\left( z\right) }{\left[ p,q%
\right] \mathcal{L}_{q}^{\mu +p-1}f\left( z\right) }=\frac{1+A\omega \left(
z\right) }{1+B\omega \left( z\right) }\text{ }\left( z\in \mathbb{D}\right) ,
\end{equation*}%
which is equalent for $z\in \mathbb{D},$ $0\leqq \theta <2\pi $%
\begin{equation}
\frac{z\partial _{q}\mathcal{L}_{q}^{\mu +p-1}f\left( z\right) }{\left[ p,q%
\right] \mathcal{L}_{q}^{\mu +p-1}f\left( z\right) }\neq \frac{1+Ae^{i\theta
}}{1+Be^{i\theta }},  \label{eq12}
\end{equation}%
and further written in more simplified form%
\begin{equation}
\left( 1+Be^{i\theta }\right) z\partial _{q}\mathcal{L}_{q}^{\mu
+p-1}f\left( z\right) -\left[ p,q\right] \left( 1+Ae^{i\theta }\right) 
\mathcal{L}_{q}^{\mu +p-1}f\left( z\right) \neq 0.  \label{eq131}
\end{equation}%
Now using the following convolution properties in $\left( \ref{eq131}\right) 
$ 
\begin{equation*}
\begin{array}{ccc}
\mathcal{L}_{q}^{\mu +p-1}f\left( z\right) \ast \frac{z^{p}}{\left(
1-z\right) }=\mathcal{L}_{q}^{\mu +p-1}f\left( z\right) & \text{and} & 
\mathcal{L}_{q}^{\mu +p-1}f\left( z\right) \ast \frac{z^{p}}{\left(
1-z\right) \left( 1-qz\right) }=z\partial _{q}\mathcal{L}_{q}^{\mu
+p-1}f\left( z\right) ,%
\end{array}%
\end{equation*}%
and then simple computation gives%
\begin{equation*}
\frac{1}{z}\left[ \mathcal{L}_{q}^{\mu +p-1}f\left( z\right) \ast \left( 
\frac{\left( 1+Be^{i\theta }\right) z^{p}}{\left( 1-z\right) \left(
1-qz\right) }-\frac{\left[ p,q\right] \left( 1+Ae^{i\theta }\right) z^{p}}{%
\left( 1-z\right) }\right) \right] \neq 0,
\end{equation*}%
or equivalently 
\begin{equation*}
\frac{\left( B-\left[ p,q\right] A\right) e^{i\theta }}{z}\left[ \mathcal{L}%
_{q}^{\mu +p-1}f\left( z\right) \ast \left( \frac{\left( N+1\right)
z^{p}-Lqz^{p+1}}{\left( 1-z\right) \left( 1-qz\right) }\right) \right] \neq
0,
\end{equation*}%
which is the required direct part.$\medskip $

\noindent Assume that $\left( \ref{eq11}\right) $ holds true for $L_{\theta
}-1=N_{\theta }=0$, it follows that 
\begin{equation*}
\frac{e^{i\theta }\left( B-\left[ p,q\right] A\right) }{z}\mathcal{L}%
_{q}^{\mu +p-1}f\left( z\right) \neq 0,\text{ for all }z\in \mathbb{D}.
\end{equation*}%
Thus the function $h\left( z\right) =\frac{z\partial _{q}\mathcal{L}%
_{q}^{\mu +p-1}f\left( z\right) }{\left[ p,q\right] \mathcal{L}_{q}^{\mu
+p-1}f\left( z\right) }$ is analytic in $\mathbb{D}$ and $h\left( 0\right)
=1.$ Since we have shown that $\left( \ref{eq131}\right) $ and $\left( \ref%
{eq11}\right) $ are equivalent, therefore we have%
\begin{equation}
\frac{z\partial _{q}\mathcal{L}_{q}^{\mu +p-1}f\left( z\right) }{\left[ p,q%
\right] \mathcal{L}_{q}^{\mu +p-1}f\left( z\right) }\neq \frac{1+Ae^{i\theta
}}{1+Be^{i\theta }}\text{ \ }\left( z\in \mathbb{D}\right) .  \label{eq16}
\end{equation}%
Suppose that%
\begin{equation*}
H\left( z\right) =\frac{1+Az}{1+Bz},\  \ z\in \mathbb{D}.
\end{equation*}%
Now from relation $\left( \ref{eq16}\right) $ it is clear that $H\left(
\partial \mathbb{D}\right) \cap h\left( \mathbb{D}\right) =\phi .$ Therefore
the simply connected domain $h\left( \mathbb{D}\right) $ is contained in a
connected component of $%
\mathbb{C}
\backslash H\left( \partial \mathbb{D}\right) .$ The univalence of the
function $h$ together with the fact $H\left( 0\right) =h\left( 0\right) =1$
shows that $h\prec H$ which shows that $f\in \mathcal{S}_{p}^{\ast }\left(
q,\mu ,A,B\right) .$
\end{proof}

\noindent We now define an integral operator for the function $f\in 
\mathfrak{A}_{p}$ as follows;

\begin{definition}
Let $f\in \mathfrak{A}_{p}.$ Then $\mathcal{L}:\mathfrak{A}_{p}\rightarrow 
\mathfrak{A}_{p}$ is called the $q$-analogue of Benardi integral operator
for multivalent functions defined by $\mathcal{L}\left( f\right) =F_{\eta
,p} $ with $\eta >-p,$ where $F_{\eta ,p}$ is given by%
\begin{eqnarray}
F_{\eta ,p}\left( z\right) &=&\frac{[\eta +p,q]}{z^{\eta }}\int
\limits_{0}^{z}t^{\eta -1}f(t)d_{q}t,  \label{B1} \\
&=&z^{p}+\sum \limits_{n=1}^{\infty }\frac{[\eta +p,q]}{[\eta +p+n,q]}%
a_{n+p}z^{n+p},\text{ }\left( z\in \mathbb{D}\right) .  \label{B2}
\end{eqnarray}
\end{definition}

\noindent We easily obtain that the series defined in $\left( \ref{B2}%
\right) $ is converges absolutely in $\mathbb{D}$. Now If $q\rightarrow 1$,
then the operator $F_{\eta ,p}$ reduces to the integral operator studied in 
\cite{Wang} and further by taking $p=1,$ we obtain the familiar Bernardi
integral operator introduced in \cite{Ber}.

\begin{theorem}
If $f$ is of the form $\left( \ref{eq1}\right) $ belongs to the family $%
\mathcal{S}_{p}^{\ast }\left( q,\mu ,A,B\right) $ and%
\begin{equation}
F_{\eta ,p}\left( z\right) =z^{p}+\sum \limits_{n=1}^{\infty }b_{n+p}z^{n+p},
\label{B3}
\end{equation}
where $F_{\eta ,p}$ is the integral operator given by $\left( \ref{B1}%
\right) ,$ then 
\begin{equation*}
\left \vert b_{p+1}\right \vert \leqq \frac{\lbrack \eta +p,q]}{[\eta +p+1,q]%
}\frac{\psi _{1}\left( A-B\right) }{\wedge _{1+p}}
\end{equation*}%
and for $n\geqq 2$%
\begin{equation*}
\left \vert b_{p+n}\right \vert \leqq \frac{\lbrack \eta +p,q]}{[\eta +p+n,q]%
}\frac{\left( A-B\right) \psi _{n}}{\wedge _{n+p}}\prod
\limits_{t=1}^{n-1}\left( 1+\frac{\left[ p,q\right] \left( A-B\right) }{%
\left( \left[ p+t,q\right] -\left[ p,q\right] \right) }\right) ,
\end{equation*}%
where $\psi _{n}$ and $\wedge _{n+p}$ are defined by $\left( \ref{eq33}%
\right) $ and $\left( \ref{ita}\right) $ respectively.
\end{theorem}

\begin{proof}
The proof follows easily by using $\left( \ref{B2}\right) $ and Theorem \ref%
{Th1}.
\end{proof}

\begin{theorem}
\label{thm3} Let $f\in \mathcal{S}_{p}^{\ast }\left( q,\mu ,A,B\right) $ and
be given by $\left( \ref{eq1}\right) .$ Also if $F_{\eta ,p}$ is the
integral operator defined by $\left( \ref{B1}\right) $ and is of the form $%
\left( \ref{B3}\right) ,$ then for $\sigma \in \mathbb{C}$%
\begin{equation*}
\left \vert b_{p+2}-\sigma b_{p+1}^{2}\right \vert \leqq \frac{\lbrack \eta
+p,q]}{[\eta +p+2,q]}\frac{(A-B)\psi _{2}}{\Lambda _{p+2}}\left \{ 1;\text{ }%
\left \vert \upsilon \right \vert \right \} ,
\end{equation*}%
where%
\begin{equation}
\upsilon =\left( B-(A-B)\psi _{1}\right) +\frac{\Lambda _{p+2}\psi _{1}^{2}}{%
\Lambda _{p+1}^{2}\psi _{2}}(A-B)\frac{[\eta +p,q][\eta +p+2,q]}{[\eta
+p+1,q]^{2}}\sigma .  \label{B4}
\end{equation}
\end{theorem}

\begin{proof}
From $\left( \ref{B2}\right) $ and $\left( \ref{B3}\right) ,$ we easily have%
\begin{eqnarray*}
b_{p+1} &=&\frac{[\eta +p,q]}{[\eta +p+1,q]}a_{p+1}, \\
b_{p+2} &=&\frac{[\eta +p,q]}{[\eta +p+2,q]}a_{p+2}.
\end{eqnarray*}%
Now 
\begin{eqnarray*}
\left \vert b_{p+2}-\sigma b_{p+1}^{2}\right \vert &=&\frac{[\eta +p,q]}{%
[\eta +p+2,q]}\left \vert a_{p+2}-\sigma \frac{\lbrack \eta +p,q][\eta
+p+2,q]}{[\eta +p+1,q]^{2}}a_{p+1}^{2}\right \vert , \\
&\leqq &\frac{[\eta +p,q]}{[\eta +p+2,q]}\frac{(A-B)\psi _{2}}{\Lambda _{p+2}%
}\left \{ 1;\text{ }\left \vert \upsilon \right \vert \right \} ,
\end{eqnarray*}%
where $\upsilon $ is given by $\left( \ref{B4}\right) $ and we have used
Theorem \ref{Th2} to complete the proof.
\end{proof}

\begin{center}
\textbf{Conflicts of Interest}
\end{center}

\noindent The authors agree with the contents of the manuscript and there
are no conflicts of interest among the authors.

\end{document}